\documentclass[12pt]{amsart}
\usepackage[margin=1.2in]{geometry}
\usepackage{graphicx,latexsym,graphics}
\usepackage{amsfonts, amssymb, amsmath, amsthm}
\usepackage{algorithmic, algorithm}

\newtheorem{theorem}{Theorem}[section]
\newtheorem{proposition}[theorem]{Proposition}

\newtheorem{lemma}[theorem]{Lemma}

\theoremstyle{definition}
\newtheorem{definition}[theorem]{Definition}

\newtheorem{remark}[theorem]{Remark}

\numberwithin{equation}{section}

\allowdisplaybreaks

\newcommand{\beq}{\begin{eqnarray}}
\newcommand{\eeq}{\end{eqnarray}}

\newcommand{\beqs}{\begin{eqnarray*}}
\newcommand{\eeqs}{\end{eqnarray*}}

\newcommand{\NN}{\mathbb N}

\newcommand{\CC}{\mathbb C}
\newcommand{\RR}{\mathbb R}
\newcommand{\ZZ}{\mathbb Z}
\newcommand{\EE}{\mathcal E}
\newcommand{\DD}{\mathcal D}
\newcommand{\SSS}{\mathcal S}

\begin{document}

\title[Wave front set with respect to Banach spaces. Characterisation via the STFT]{Wave front sets with respect to Banach spaces of ultradistributions. Characterisation via the short-time Fourier transform}
\thanks{The research was
partially supported by the COST Action CA15225
``Fractional-- order
systems--analysis, synthesis and
their importance for future design''
and by the
bilateral project
``Microlocal analysis and applications''
between the Macedonian and
Serbian academies of sciences and arts.}

\author[P. Dimovski]{Pavel Dimovski}
\address{P. Dimovski, Faculty of Technology and Metallurgy, University Ss. Cyril and Methodius, Ruger Boskovic 16, 1000 Skopje, Macedonia}
\email{dimovski.pavel@gmail.com}

\author[B. Prangoski]{Bojan Prangoski}
\address{B. Prangoski, Faculty of Mechanical Engineering, University Ss. Cyril and Methodius, Karpos II bb, 1000 Skopje, Macedonia}
\email{bprangoski@yahoo.com}



\begin{abstract}
We define ultradistributional wave front sets with respect to translation-modulation invariant Banach spaces of ultradistributions having solid Fourier image. The main result is their characterisation by the short-time Fourier transform.
\end{abstract}

\maketitle

\section{Introduction}

H\"ormander \cite{hor0} introduced the Sobolev wave front set of a distribution $f$ as the set of points $x$ and directions $\xi$ at which $f$ does not behave as an element of the Sobolev space $H^s$; i.e. it is not $H^s$ micro-regular at $(x,\xi)$. It is one of the most powerfull tools in studying the regularity of solutions of PDEs with wide range of applications in mathematical physics. Many authors considered various generalisations and characterisations of the Sobolev wave front set and other similar variants; see \cite{nakamura,piltt,rw,rzt,ScW,ScW2} and the references there in. This concept was further generalised recently in \cite{coriasco1,coriasco2} where the wave front set is defined with respect to a general Banach spaces of distributions satisfying appropriate assumptions. In the setting of ultradistributions, the wave front set with respect to Fourier-Lebesgue spaces having sub-exponential weights was considered in \cite{anj,JPTT} where the authors also gave a discrete characterisation of it.\\
\indent The goal of this article is to define the wave front set in the setting of non quasi-analytic ultradistributions with respect to a Banach space of ultradistributions satisfying appropriate assumptions; this generalisation is in the spirit of \cite{coriasco1}, where the distributional case was considered. The main result of the article (Theorem \ref{svlmtk157}) is its characterisation by the short-time Fourier transform (cf. \cite{wf-pp} for a similar characterisation of the Sobolev wave front set in the distributional setting).

\section{Preliminaries}

Let $M_p$, $p\in\NN$, be a sequence of positive numbers satisfying
$M_0=M_1=1$ for which the following conditions hold true:\\
\indent $(M.1)$ $M_p^2\leq M_{p-1}M_{p+1}$, $p\in\ZZ_+$;\\
\indent $(M.2)$ there exist $c_0,H\geq 1$ such that $M_p\leq c_0
H^p\min_{0\leq q\leq p}M_qM_{p-q}$, $p\in\NN$;\\
\indent $(M.3)$ there exists $c_0\geq 1$ such that
$\sum_{q=p+1}^{\infty}M_{q-1}/M_q\leq c_0pM_p/M_{p+1}$,
$p\in\ZZ_+$;\\
\indent $(M.4)$ $M^2_p/p!^2\leq (M_{p-1}/(p-1)!)\cdot
(M_{p+1}/(p+1)!)$.\\
The sequence $M_p=p!^s$, $s>1$, satisfies all of the above conditions. When $\alpha\in\NN^d$, we set $M_{\alpha}=M_{|\alpha|}$. The associated function to the sequence $M_p$ is defined by $M(\lambda)=\sup_{p\in\NN}\ln_+(\lambda^p/M_p)$, $\lambda>0$ (see \cite{Komatsu1}). It is continuous, non-negative, monotonically increasing function, it vanishes for sufficiently small $\lambda>0$ and increases more rapidly than $\ln \lambda^p$ as $\lambda$ tends to infinity for any $p\in\NN$.\\
\indent Given an open set $U\subseteq \RR^d$, we refer to Komatsu \cite{Komatsu1} for the definition and the basic properties of the locally convex spaces (from now on abbreviated as l.c.s.) $\EE^{(M_p)}(U)$ and $\EE^{\{M_p\}}(U)$ of ultradifferentiable functions of Beurling and Roumieu type respectively, as well as the corresponding spaces $\DD^{(M_p)}(U)$ and $\DD^{\{M_p\}}(U)$ of ultradifferentiable functions having compact support in $U$. Their strong duals are the corresponding spaces of ultradistributions of Beurling and Roumieu type. We also denote by $\DD^{(M_p)}_K$ and $\DD^{\{M_p\}}_K$ the spaces consisting of all elements of $\EE^{(M_p)}(U)$ and $\EE^{\{M_p\}}(U)$ respectively, supported by the compact set $K\subset U$ (cf. \cite{Komatsu1}). The common notation for the symbols $(M_{p})$ and $\{M_{p}\} $ will be $*$. If $\varphi\in\EE^*(U)$ never vanishes than $1/\varphi$ also belongs to $\EE^*(U)$. More precisely, we have the following result.

\begin{lemma}\label{invcc}
Let $\varphi\in\EE^*(U)$ and let $V\subseteq U$ be the open set where $\varphi\neq 0$. The function $x\mapsto 1/\varphi(x)$ belongs to $\EE^*(V)$.
\end{lemma}

\begin{proof} The proof relies on the multidimensional Fa\'a di Bruno formula \cite[Corollary 2.10]{faadib} applied to the composition of the functions $\lambda\mapsto 1/\lambda$ and $\varphi$ and the condition $(M.4)$ on $M_p$; it is similar to the proof of \cite[Lemma 7.5]{compo} and we omit it (see \cite[Theorem 4.1]{bruna} for the one dimensional Beurling case and \cite[Theorem 3]{siddiqi} for the one dimensional Roumieu case).
\end{proof}

The entire function $P(z) =\sum _{\alpha \in \NN^d}c_{\alpha } z^{\alpha}$, $z \in \CC^d$, is an ultrapolynomial of class $(M_{p})$ (resp. of class $\{M_{p}\}$), whenever the coefficients $c_{\alpha }$ satisfy the estimate $|c_{\alpha }|  \leq C L^{|\alpha| }/M_{\alpha}$, $\alpha \in \NN^d$, for some $C,L > 0$ (resp. for every $L > 0$ and some $C=C(L) > 0$). The corresponding operator $P(D)=\sum_{\alpha} c_{\alpha}D^{\alpha}$ is called an ultradifferential operator of class $(M_{p})$ (resp. of class $\{M_{p}\}$) and it acts continuously on $\EE^{(M_p)}(U)$ and $\DD^{(M_p)}(U)$ (resp. on $\EE^{\{M_p\}}(U)$ and $\DD^{\{M_p\}}(U)$) and the corresponding spaces of ultradistributions.\\
\indent The Fourier transform of $f\in L^1(\RR^d)$ is given by $\mathcal{F}f (\xi)=\int_{\RR^d} e^{-i x\xi}f(x)dx$, $\xi\in\RR^d$.\\
\indent For $m>0$, we denote by $\SSS^{M_p,m}(\RR^d)$ the $(B)$-space of all $\varphi\in\mathcal{C}^{\infty}(\RR^d)$ for which the norm $\|\varphi\|_m=\sup_{\alpha\in \NN^d}m^{|\alpha|}\|e^{M(m|\cdot|)}D^{\alpha}\varphi\|_{L^{\infty}(\RR^d)}/M_{\alpha}$ is finite. The spaces of sub-exponentially decreasing ultradifferentiable functions of Beurling and Roumieu type are defined by
\beqs
\SSS^{(M_{p})}(\RR^d)=\lim_{\substack{\longleftarrow\\ m\rightarrow\infty}}\SSS^{M_{p},m}(\RR^d)\,\, \mbox{and}\,\, \SSS^{\{M_{p}\}}(\RR^d)=\lim_{\substack{\longrightarrow\\ m\rightarrow 0}}\SSS^{M_{p},m}(\RR^d),
\eeqs
respectively and their strong duals $\SSS'^{(M_{p})}(\RR^d)$ and $\SSS'^{\{M_{p}\}}(\RR^d)$ are the spaces of tempered ultradistributions of Beurling and Roumieu type, respectively. When $M_p=p!^s$, $s>1$, $\SSS^{\{M_p\}}(\RR^d)$ is just the Gelfand-Shilov space $\SSS^s_s(\RR^d)$. The ultradifferential operators of class $*$ act continuously on $\SSS^*(\RR^d)$ and $\SSS'^*(\RR^d)$ and the Fourier transform is a topological isomorphism on them. We refer to \cite{PilipovicK} for the topological properties of $\SSS^*(\RR^d)$ and $\SSS'^*(\RR^d)$.\\
\indent For $f\in\SSS'^*(\RR^d)$ and $0\neq\chi\in\SSS^*(\RR^d)$, the short-time Fourier transform of $f$ with window $\chi$ (from now on abbreviated as STFT \cite{groch}; it is also known as the wave-packet transform first introduced by C\'ordoba and Fefferman \cite{corff}) is defined by $V_{\chi}f(x,\xi)=\mathcal{F}_{t\rightarrow \xi}(f(t)\overline{\chi(t-x)})$. For fixed window $\chi$, $f\mapsto V_{\chi}f$ is a continuous operator from $\SSS'^*(\RR^d)$ into $\SSS'^*(\RR^{2d})$ and it restricts to a continuous operator from $\SSS^*(\RR^d)$ into $\SSS^*(\RR^{2d})$. Furthermore, when $f\in\SSS'^*(\RR^d)$, $V_{\chi}f$ is smooth and, in fact, it is an element of $\EE^*(\RR^{2d})$. If the window $\chi$ is in $\DD^*(\RR^d)$, we can extend the definition of $V_{\chi}f$ even when $f\in\DD'^*(\RR^d)$ by $V_{\chi}f(x,\xi)=\langle e^{-i\xi \cdot} f,\overline{\chi(\cdot-x)}\rangle$ and one can easily verify that $V_{\chi}f\in\EE^*(\RR^{2d})$ in this case as well (see Remark \ref{ask15} below).\\
\indent We denote by $T_x$ and $M_{\xi}$ the translation and modulation operators: $T_x f= f(\cdot-x)$, $M_{\xi}f= e^{i\xi\cdot} f(\cdot)$. They act continuously on $\SSS^*(\RR^d)$ and, by duality, on $\SSS'^*(\RR^d)$ as well.\\
\indent We end the section by recalling the definition and some of the important properties of translation-modulation invariant $(B)$-spaces of ultradistributions \cite{DPPV1}.

\begin{definition}\label{defmod}(\cite[Definition 3.1]{DPPV1})
A $(B)$-space $E$ is said to be a translation-modulation invariant $(B)$-space of ultradistributions (in short: TMIB)) of class $*$ if it satisfies the following three conditions:
\begin{itemize}
    \item[$(a)$] The continuous and dense inclusions $\mathcal{S}^*(\mathbb{R}^d)\hookrightarrow E\hookrightarrow \mathcal{S}'^*(\RR^d)$ hold.
    \item[$(b)$] $T_{x}(E)\subseteq E$ and $M_\xi (E)\subseteq E$ for all $x,\xi\in\mathbb{R}^d$.
\item [$(c)$] There exist $\tau,C>0$ (for every $\tau>0$ there exists $C_{\tau}>0$), such that\footnote{The closed graph theorem together with the conditions (a) and (b) yield that $T_x,M_{\xi}\in\mathcal{L}(E)$, for all $x,\xi\in\RR^d$ (see the proof of \cite[Lemma 3.1]{TIBU}); hence, we can take their operator norms in \eqref{omega}.}
    \begin{equation}\label{omega}
   \omega_{E}(x):= \|T_{x}\|_{\mathcal{L}_b(E)}\leq C e^{M(\tau|x|)} \quad  \mbox{and} \quad \nu_{E}(\xi) := \|M_{-\xi}\|_{\mathcal{L}_b(E)}\leq C e^{M(\tau|\xi|)},
   \end{equation}
   where $\|\cdot \|_{\mathcal{L}_b(E)}$ stands for the norm on $\mathcal{L}(E)=\mathcal{L}(E,E)$ induced by $\|\cdot\|_E$ (the norm on $E$).
\end{itemize}
The functions $\omega_{E}:\mathbb{R}^{d}\to (0,\infty)$ and $\nu_{E}:\mathbb{R}^{d}\to (0,\infty)$ defined in \eqref{omega} are called the weight functions of the translation and modulation groups of $E$, respectively (in short: its weight functions).
\end{definition}

These spaces enjoy a number of important properties; we recall only the necessary ones here and refer to \cite{DPPV1} for the complete account (see also \cite{DPPV,TIBU}). We start by pointing out that $E$ is separable and the weight functions $\omega_E$ and $\nu_E$ are measurable. Moreover, the translation and modulation operators on $E$ form both $C_0$-groups, i.e. $x\mapsto T_x f$ and $x\mapsto M_x f$, $\RR^d\mapsto E$, are continuous for each $f\in E$. Also $E$ is a Banach convolution module over the Beurling (convolution) algebra $L^1_{\omega_E}(\RR^d)$ (the weighted $L^1$ space of measurable functions $g$ such that $\|g\|_{L^1_{\omega_E}}:=\|g\omega_E\|_{L^1}<\infty$) and a Banach multiplication module over the Wiener-Beurling (multiplication) algebra $\mathcal{F}L^1_{\nu_E}$ (see \cite[Proposition 3.2]{DPPV1}). In particular, multiplication by elements of $\SSS^*(\RR^d)$ is a well defined and continuous operation on $E$. Furthermore, the Fourier image of $E$, which we denote by $\mathcal{F}E$, is again a TMIB space of class $*$ with norm $\|\mathcal{F}f\|_{\mathcal{F}E}=\|f\|_E$ and, consequently, it enjoys all of the properties we mentioned above; in particular, its weight functions $\omega_{\mathcal{F}E}$ and $\nu_{\mathcal{F}E}$ are measurable and satisfy the estimate (\ref{omega}), with $E$ replaced by $\mathcal{F}E$.

\section{The wave front set with respect to a TMIB space of class $*$. Characterisations via the STFT}

Let $E$ be TMIB space of class $*$ over $\RR^d$. Besides the properties $(a)$, $(b)$ and $(c)$ of Definition \ref{defmod} we additionally assume that it satisfies the following:
\begin{itemize}
\item[$(d)$] $\mathcal{F}E$ is a solid space (cf. \cite{F84}), i.e. $\mathcal{F}E\subseteq L^1_{\mathrm{loc}}(\RR^d)$ \footnote{Since $\mathcal{F}E$ is continuously included into $\DD'^*(\RR^d)$, the closed graph theorem for Fr\'echet spaces immediately implies that the inclusion $\mathcal{F}E\subseteq L^1_{\mathrm{loc}}(\RR^d)$ is continuous.} and there exists $C_0>0$ such that if $g\in L^1_{\mathrm{loc}}(\RR^d)$, $f\in\mathcal{F}E$ and $|g(x)|\leq |f(x)|$ a.e. then $g\in\mathcal{F}E$ and $\|g\|_{\mathcal{F}E}\leq C_0\|f\|_{\mathcal{F}E}$.
\end{itemize}
Notice that the solidity implies that if $f\in \mathcal{F}E$ then $|f|\in\mathcal{F}E$ and $\||f|\|_{\mathcal{F}E}\leq C_0\|f\|_{\mathcal{F}E}$. Consequently, if $g\in L^1_{\mathrm{loc}}(\RR^d)$ and $f_1,\ldots, f_k\in\mathcal{F}E$ are such that $|g(x)|\leq \sum_{j=1}^k |f_j(x)|$ a.e. then $g\in \mathcal{F}E$ and $\|g\|_{\mathcal{F}E}\leq C^2_0\sum_{j=1}^k \|f_j\|_{\mathcal{F}E}$.\\
\indent Following H\"ormander \cite{hor11} (cf. \cite[Section 8.1, p. 253]{hor}), for $f\in\EE'^*(\RR^d)$ we define
the set $\Sigma_E(f)\subseteq\RR^d\backslash\{0\}$ as follows:
$\xi\in\RR^d\backslash\{0\}$ does not belong to $\Sigma_E(f)$ if
and only if there exists a cone neighbourhood $\Gamma$ of $\xi$
such that
\beq\label{defofwfe}
\theta_{\Gamma}\mathcal{F}f\in\mathcal{F}E,
\eeq
where $\theta_{\Gamma}$ denotes the characteristic function of
$\Gamma$. Clearly $\Sigma_E(f)$ is a closed cone in
$\RR^d\backslash\{0\}$. From now on, for a measurable subset $G\subseteq \RR^d$, $\theta_G$ will always
stand for the characteristic function of $G$.

\begin{remark}\label{ask15}
Before we state the next result we make the following general
observation. For every $f\in\EE'^*(\RR^d)$, $\mathcal{F}f\in \EE^*(\RR^d)$. Furthermore, if $B$ is a bounded subset of $\EE'^*(\RR^d)$ then there exist $C_1,h_1>0$ (resp. for every $h_1>0$ there exists $C_1>0$)
such that $|\mathcal{F}f(x)|\leq C_1 e^{M(h_1|x|)}$, $\forall
x\in\RR^d$, $\forall f\in B$. This easily follows from \cite[Proposition 5.11]{Komatsu1} and \cite[Theorem 8.1
and Theorem 8.7]{Komatsu1}.\\
\indent On the other hand, if $f\in\EE'^*(\RR^d)$ satisfies the following estimate: for every $h_1>0$ there exists $C_1>0$ (resp. there exist $h_1,C_1>0$) such that $|\mathcal{F}f(x)|\leq C_1 e^{-M(h_1|x|)}$, $\forall x\in\RR^d$, then a straightforward computation gives $f\in\DD^*(\RR^d)$.
\end{remark}

We recall the following lemmas from \cite{ppv} which will be used in the
proof of Proposition \ref{pro1}; we state only the Beurling case of these results since this is the only part we need for the proof of Proposition \ref{pro1}.

\begin{lemma}\label{lemma1}\cite[Lemma 2.1]{ppv}
Let $r'\geq 1$ and $k > 0$. There exists an ultrapolynomial
$P(z)$ of class $(M_p)$ such that $P$ does not vanish on $\RR^d$ and satisfies
the following estimate: there exists $C > 0$ such that
\beqs
|D^\alpha(1/P(x))|\leq C \alpha!r'^{-|\alpha|}e^{-M(k|x|)},\,\,\, \forall x\in \RR^d,\, \forall\alpha\in\NN^d.
\eeqs
\end{lemma}

\begin{lemma}\label{lemma2}\cite[Lemma 2.4]{ppv}
Let $r > 0$.
\begin{itemize}
\item[$(i)$] For each $\chi,\varphi\in \SSS^{(M_p)}(\RR^d)$ and $\psi\in \SSS^{M_p,r}(\RR^d)$ it holds that $\chi*(\varphi\psi)\in \SSS^{(M_p)}(\RR^d)$.
\item[$(ii)$] Let $\varphi,\chi \in \SSS^{(M_p)}(\RR^d)$ with $\varphi(0) = 1$ and $\int_{\RR^d}\chi(x)dx=1$. For each $n\in \mathbb{Z}_+$ define
$\chi_n(x) = n^d\chi(nx)$ and $\varphi_n(x) = \varphi(x/n)$. Then there exists $k\geq 2r$ such that the operators $\tilde{Q}_n : \psi\mapsto\chi_n* (\varphi_n\psi)$ are continuous
as mappings from $\SSS^{M_p,k}(\RR^d)$ to $\SSS^{M_p,r}(\RR^d)$, for all $n\in\mathbb{ Z}_+$. Moreover $\tilde{Q}_n\rightarrow{\rm Id}$, as $n\rightarrow \infty$, in $\mathcal{L}_b(\SSS^{M_p,k}(\RR^d),\SSS^{M_p,r}(\RR^d))$.
\end{itemize}
\end{lemma}

\begin{proposition}\label{pro1}
Let $\psi\in\mathcal{D}^*(\mathbb{R}^d)$ and $f\in\EE'^*(\RR^d)$.
Then $\Sigma_E(\psi f)\subseteq \Sigma_E(f)$.
\end{proposition}

\begin{proof} Let $0\neq \xi_0\not\in \Sigma_E(f)$. There exists a cone neighbourhood $\Gamma_1$ of $\xi_0$ such that (\ref{defofwfe}) holds.
Pick a cone neighbourhood $\Gamma$ of $\xi_0$ such that
$\overline{\Gamma}\subseteq\Gamma_1\cup\{0\}$. We have
\beqs
\theta_{\Gamma}(\xi)\mathcal{F}(\psi
f)(\xi)=(2\pi)^{-d}\theta_{\Gamma}(\xi)\mathcal{F}\psi*\mathcal{F}f(\xi)=I_1(\xi)/(2\pi)^d+I_2(\xi)/(2\pi)^d,
\eeqs
where
\beqs
I_1(\xi)&=&\theta_{\Gamma}(\xi)\int_{\RR^d}\mathcal{F}\psi(\eta)(1-\theta_{\Gamma_1}(\xi-\eta))\mathcal{F}f(\xi-\eta)d\eta,\\
I_2(\xi)&=&\theta_{\Gamma}(\xi)\int_{\RR^d}\mathcal{F}\psi(\eta)\theta_{\Gamma_1}(\xi-\eta)\mathcal{F}f(\xi-\eta)d\eta.
\eeqs
Clearly $I_1,I_2\in L^1_{\mathrm{loc}}(\RR^d)$. We prove
that both $I_1$ and $I_2$ belong to $\mathcal{F}E$ which, in turn, will yield the claim in the proposition. For this
purpose we make the following observations: there exist $0<c<1$ such
that
\beq\label{ine11}
\{\eta\in
\mathbb{R}^d|\,\exists\xi\in\Gamma,\, |\eta-\xi|\leq
c|\xi|\}\subseteq \Gamma_1.
\eeq
For $I_1$, we avail ourselves of
(\ref{ine11}) by noticing that if $\xi\in \Gamma$ and
$\xi-\eta\not\in\Gamma_1$ than $|\eta|>c|\xi|$. Thus, applying
Remark \ref{ask15} together with \cite[Proposition 3.6]{Komatsu1} we infer
\beqs
e^{M(h|\xi|)}|I_1(\xi)|&\leq&
C_1\int_{\mathbb{R}^d}|\mathcal{F}\psi(\eta)|e^{M(h|\eta|/c)}e^{M(h_1(1+c^{-1})|\eta|)}d\eta\\
&\leq&c_0C_1\int_{\mathbb{R}^d}|\mathcal{F}\psi(\eta)|e^{M((hc^{-1}+h_1c^{-1}+h_1)H|\eta|)}d\eta.
\eeqs
Hence $e^{M(h|\cdot|)}I_1\in L^{\infty}(\RR^d)$ for every
$h>0$ in the $(M_p)$ case and for some $0<h\leq 1$ in the
$\{M_p\}$ case (in the $\{M_p\}$ case we can take $h_1$
arbitrarily small in the above estimates). Because of Lemma
\ref{lemma1}, we can find an ultrapolynomial $P(z)$ of
class $(M_p)$ which does not vanish on the real axis and satisfies
the following estimate: there exists $C'>0$ such that
$|D^{\alpha}(1/P(x))|\leq C'\alpha! e^{-M(|x|)}$, $\forall
x\in\RR^d$. As $P$ is of class $(M_p)$, there exist
$\tilde{C},s\geq 1$ such that $|P(x)|\leq \tilde{C}e^{M(s|x|)}$,
$\forall x\in\RR^d$ (see \cite[Proposition 4.5]{Komatsu1}). Thus, in the $\{M_p\}$ case, $x\mapsto
1/P(hx/s)$ belongs to $\SSS^{\{M_p\}}(\RR^d)$ and $|I_1(\xi)|\leq
C''/|P(h\xi/s)|$, $\forall \xi\in\RR^d$, for some $C''>0$. The
solidity of $\mathcal{F}E$ gives $I_1\in \mathcal{F}E$ in the
$\{M_p\}$ case. For the $(M_p)$ case, since $\SSS^{(M_p)}(\RR^d)$
is continuously included into $\mathcal{F}E$, there exist
$C',h'\geq 1$ such that $\|\varphi\|_{\mathcal{F}E}\leq C'
\|\varphi\|_{h'}$, $\forall \varphi\in\SSS^{(M_p)}(\RR^d)$, which
in turn yields that the closure of $\SSS^{(M_p)}(\RR^d)$ in
$\SSS^{M_p,h'}(\RR^d)$, which we denote by $X_{h'}$ for short, is
continuously included into $\mathcal{F}E$. Now Lemma \ref{lemma2} gives the existence of $h''>h'$ such that (in the notations of Lemma \ref{lemma2}) $\tilde{Q}_n\phi\in \SSS^{(M_p)}(\RR^d)$, $\forall\phi\in \SSS^{M_p,h''}(\RR^d)$, and $\tilde{Q}_n\phi\rightarrow \phi$, as $n\rightarrow \infty$, in the topology of $\SSS^{M_p,h'}(\RR^d)$. We conclude that $\SSS^{M_p,h''}(\RR^d)$ is continuously included into $X_{h'}$. If
$P(z)$ is the same ultrapolynomial as before, than $x\mapsto
1/P(h''x)$ belongs to $\SSS^{M_p,h''}(\RR^d)$ (notice that $(M.3)$
gives $h''^{|\alpha|}\alpha!\leq C_2h''^{-|\alpha|}M_{\alpha}$,
$\forall \alpha\in\NN^d$, for some $C_2>0$) and consequently in
$\mathcal{F}E$ as well. Since $|I_1(\xi)|\leq C''/|P(h''\xi)|$,
$\forall\xi\in\RR^d$, the solidity of $\mathcal{F}E$ proves that
$I_1\in\mathcal{F}E$ in the $(M_p)$ case as well.\\
\indent We turn our attention to $I_2$ next. Let $\eta\in
\mathbb{R}^d$ be fixed. Then
\beqs
\theta_{\Gamma}(\xi)\theta_{\Gamma_1}(\xi-\eta)|\mathcal{F}f(\xi-\eta)|\leq|T_{\eta}(\theta_{\Gamma_1}\mathcal{F}f)(\xi)|,\,\,
\forall \xi\in\RR^d.
\eeqs
Since
$\theta_{\Gamma_1}\mathcal{F}f\in\mathcal{F}E$, the solidity of
$\mathcal{F}E$ implies that $\eta\mapsto
\mathbf{F}(\eta)=\theta_{\Gamma}T_{\eta}(\theta_{\Gamma_1}\mathcal{F}f)$,
$\RR^d\rightarrow \mathcal{F}E$, is well defined
$\mathcal{F}E$-valued mapping and
\beq\label{inequal13}
\|\mathbf{F}(\eta)\|_{\mathcal{F}E}\leq C_0
\omega_{\mathcal{F}E}(\eta)\|\theta_{\Gamma_1}\mathcal{F}f\|_{\mathcal{F}E}.
\eeq
For $\eta,\eta_0\in\RR^d$, we have
\beqs
|\theta_{\Gamma}(\xi)T_{\eta}(\theta_{\Gamma_1}\mathcal{F}f)(\xi)-\theta_{\Gamma}(\xi)
T_{\eta_0}(\theta_{\Gamma_1}\mathcal{F}f)(\xi)| \leq
|T_{\eta}(\theta_{\Gamma_1}\mathcal{F}f)(\xi)-T_{\eta_0}(\theta_{\Gamma_1}\mathcal{F}f)(\xi)|,\,\,\,
\forall\xi\in\RR^d.
\eeqs
Again, the solidity of $\mathcal{F}E$ implies
\beqs
\|\mathbf{F}(\eta)-\mathbf{F}(\eta_0)\|_{\mathcal{F}E}\leq C_0
\|T_{\eta}(\theta_{\Gamma_1}\mathcal{F}f)-T_{\eta_0}(\theta_{\Gamma_1}\mathcal{F}f)\|_{\mathcal{F}E}\rightarrow
0,\,\, \mbox{as}\,\, \eta\rightarrow\eta_0.
\eeqs
Consequently, $\mathbf{F}$ is continuous and hence strongly measurable. Now,
(\ref{inequal13}) proves that $\eta\mapsto
\mathcal{F}\psi(\eta)\mathbf{F}(\eta)$, $\RR^d\rightarrow
\mathcal{F}E$, is Bochner integrable. We claim
\beq\label{srtkch135}
I_2=\int_{\mathbb{R}^d}\mathcal{F}\psi(\eta)\mathbf{F}(\eta)d\eta\in\mathcal{F}E.
\eeq
To verify this, fix $\varphi\in\mathcal{D}^*(\mathbb{R}^d)$.
Then
\beqs
\left\langle\int_{\mathbb{R}^d}\mathcal{F}\psi(\eta)\mathbf{F}(\eta)d\eta,\varphi\right\rangle&=&
\int_{\mathbb{R}^d}\mathcal{F}\psi(\eta)\langle
\mathbf{F}(\eta),\varphi\rangle
d\eta\\
&=&\int_{\RR^{2d}}\mathcal{F}\psi(\eta)\theta_{\Gamma}(\xi)\theta_{\Gamma_1}(\xi-\eta)\mathcal{F}f(\xi-\eta)\varphi(\xi)d\xi
d\eta,
\eeqs
where, the very last integral is absolutely
convergent. We conclude
\beqs
\left\langle\int_{\mathbb{R}^d}\mathcal{F}\psi(\eta)\mathbf{F}(\eta)d\eta,\varphi\right\rangle=\langle
I_2,\varphi\rangle.
\eeqs
As $\varphi\in\DD^*(\RR^d)$ is
arbitrary, we deduce (\ref{srtkch135}), which completes the
proof of the proposition.
\end{proof}

Following H\"ormander \cite{hor11} (cf. \cite[Section 8.1, p. 253]{hor}), for $f\in\DD'^*(\RR^d)$ and $x\in\RR^d$, we define
\beqs
\Sigma_{x,E}(f)=\bigcap_{\chi\in\DD^*(\RR^d),\, \chi(x)\neq 0}\Sigma_E(\chi f).
\eeqs
Clearly, $\Sigma_{x,E}(f)$ is a closed cone subset of $\RR^d\backslash\{0\}$.

\begin{proposition}\label{reg}
Let $f\in\DD'^*(\RR^d)$, $x\in\RR^d$ and $\Gamma$ be an open cone such that
$\Sigma_{x,E}(f)\subseteq \Gamma$. There exists
$\chi\in\DD^*(\RR^d)$ satisfying $\chi(x)\neq 0$ and having
support arbitrarily close to $x$ such that $\Sigma_E(\chi
f)\subseteq \Gamma$. In particular, $\Sigma_{x,E}(f)=\emptyset$ if
and only if there exists $\chi\in\DD^*(\RR^d)$, satisfying
$\chi(x)\neq 0$, such that $\chi f\in E$.
\end{proposition}

\begin{proof} The proof is the same as in \cite[Section 8.1, p. 253-254]{hor} but now applying Proposition \ref{pro1} and Lemma \ref{invcc}.
\end{proof}

We can now define the wave front set of $f\in\DD'^*(\RR^d)$ with respect to $E$.

\begin{definition}\label{def1}
For $f\in\DD'^*(\RR^d)$, we define the $E$-wave front set of $f$ by
\beqs WF_E(f)=\{(x,\xi)\in\RR^d\times(\RR^d\backslash\{0\})|\,
\xi\in\Sigma_{x,E}(f)\}.
\eeqs
\end{definition}

\begin{remark}
Clearly, $WF_E(f)$ is a closed subset of
$\RR^d\times(\RR^d\backslash\{0\})$ and it is conic in the second
variable, i.e. if $(x,\xi)\in WF_E(f)$ than $(x, \lambda\xi)\in
WF_E(f)$, $\forall \lambda>0$. Hence, we can consider it as a
closed subspace of $\RR^d\times\mathbb{S}^{d-1}$.
\end{remark}

\begin{remark}
When $E$ is a Sobolev space $H^s(\RR^d)$, $s\in\RR$, and $f\in\DD'(\RR^d)$, then the definition of $WF_E(f)$ coincides with the Sobolev wave front set of $f$ as defined by H\"ormander \cite[Definition 8.2.5, p. 188; Proposition 8.2.6, p. 189]{hor0}.
\end{remark}

\begin{remark}
For $f\in\DD'^*(\RR^d)$, we can define the set $\mathrm{sing}\,\, \mathrm{supp}_Ef\subseteq \RR^d$ whose complement is given by the points at which $f$ locally behaves as an element of $E$. More precisely, $x_0\in\RR^d$ does not belong to $\mathrm{sing}\,\, \mathrm{supp}_E f$ if and only if there exists $\chi\in\DD^*(\RR^d)$ satisfying $\chi(x_0)\neq0$ such that $\chi f\in E$ (because of Lemma \ref{invcc}, this is the same as if we furthermore require for $\chi$ to be identically equal to $1$ on a neighbourhood of $x_0$). Clearly, $\mathrm{sing}\,\, \mathrm{supp}_E f$ is closed in $\RR^d$ and Proposition \ref{reg} proves that the projection of $WF_E(f)$ on the first component is exactly $\mathrm{sing}\,\, \mathrm{supp}_E f$.
\end{remark}

We can now formulate and prove the main result of the article.

\begin{theorem}\label{svlmtk157}
Let $f\in\mathcal{D}'^*(\RR^d)$ and $(x_0,\xi_0)\in
\mathbb{R}^d\times(\mathbb{R}^d\backslash\{0\})$. The following
conditions are equivalent.
\begin{itemize}
\item[$(i)$] $(x_0,\xi_0)\not\in WF_E(f)$.
\item[$(ii)$] There exist a cone neighbourhood $\Gamma$ of $\xi_0$ and a compact neighbourhood $K$ of $x_0$ such that the mapping
$\chi\mapsto\theta_{\Gamma}\mathcal{F}(\chi f)$,
$\mathcal{D}^*_K\rightarrow \mathcal{F}E$, is well-defined and
continuous.
\item[$(iii)$] There exist a cone neighbourhood $\Gamma$ of $\xi_0$ and a compact neighbourhood $K$ of $x_0$ such that
\beqs
\theta_{\Gamma}V_{\chi}f(x,\cdot)\in\mathcal{F}E,\,\,\,
\forall \chi\in\DD^*_{K-\{x_0\}},\,\, \forall x\in K,
\eeqs
the mapping $x\mapsto \theta_{\Gamma} V_{\chi}f(x,\cdot)$,
$K\rightarrow \mathcal{F}E$, is continuous and the mapping
\beq\label{vsrkln157}
\chi\mapsto \theta_{\Gamma}V_{\bar{\chi}} f,\,\,\,
\DD^*_{K-\{x_0\}}\rightarrow
\mathcal{C}(K;\mathcal{F}E),
\eeq
is continuous.\footnote{$\mathcal{C}(K;\mathcal{F}E)$ stands for the $(B)$-space of all continuous functions $K\rightarrow \mathcal{F}E.$}
\item[$(iv)$] There exist a cone neighbourhood $\Gamma$ of $\xi_0$, a compact neighbourhood $K$ of $x_0$ and
$\chi\in\DD^*(\RR^d)$, satisfying $\chi(0)\neq 0$ such that
$\theta_{\Gamma}V_{\chi}f(x,\cdot)\in\mathcal{F}E$, $\forall x\in
K$, and $\sup_{x\in K}\|\theta_{\Gamma}V_{\chi}f(x,\cdot)\|_{\mathcal{F}E}<\infty$.
\end{itemize}
\end{theorem}

\begin{proof}
$(i)\Rightarrow (ii)$. Pick a cone neighbourhood $\Gamma_1$ of
$\xi_0$ and $\chi\in\mathcal{D}^*(\RR^d)$ with $\chi(x_0)\neq 0$
such that $\theta_{\Gamma_1}\mathcal{F}(\chi f)\in\mathcal{F}E$.
Let $K_1$ be a compact neighbourhood of $x_0$ such that $\chi$
never vanishes on $K_1$ and take a compact neighbourhood $K$ of
$x_0$ such that $K\subset \mathrm{int}\, K_1$. Fix an open cone
$\Gamma\ni\xi_0$ satisfying
$\overline{\Gamma}\subseteq\Gamma_1\cup\{0\}$ and find $0<c< 1$
such that (\ref{ine11}) holds true. Repeating the proof of
Proposition \ref{pro1} verbatim with $\chi f\in\EE'^*(\RR^d)$ in
place of $f$ we conclude that $\theta_{\Gamma}\mathcal{F}(\psi\chi
f)\in\mathcal{F}E$, for all $\psi\in\DD^*_K$. Lemma \ref{invcc} infers the function
$x\mapsto 1/\chi(x)$, $\mathrm{int}\, K_1\rightarrow \CC$, belongs
to $\EE^*(\mathrm{int}\, K_1)$. Thus, for each $\psi\in\DD^*_K$, we have $\psi f=(\psi/\chi)\chi
f$, with $\psi/\chi\in\DD^*_K$. We deduce that
$\theta_{\Gamma}\mathcal{F}(\psi f)\in\mathcal{F}E$, for all
$\psi\in\DD^*_K$. Since $\psi\mapsto
\theta_{\Gamma}\mathcal{F}(\psi f)$, $\DD^*_K\rightarrow
\SSS'^*(\RR^d)$, is continuous we conclude that the mapping
$\psi\mapsto \theta_{\Gamma}\mathcal{F}(\psi f)$,
$\DD^*_K\rightarrow \mathcal{F}E$, has closed graph
($\mathcal{F}E$ is continuously included into $\SSS'^*(\RR^d)$).
The Ptak closed graph theorem \cite[Theorem 8.5, p. 166]{Sch} implies that it
is continuous ($\DD^*_K$ is barrelled and $\mathcal{F}E$ is a $(B)$-space and consequently a Ptak space; see \cite[Section 4.8, p. 162]{Sch}).\\
$(ii)\Rightarrow (iii)$. Let $K_1$ be a compact neighbourhood of
$x_0$ and $\Gamma$ a cone neighbourhood of $\xi_0$ such that
$\chi\mapsto \theta_{\Gamma}\mathcal{F}(\chi f)$,
$\DD^*_{K_1}\rightarrow\mathcal{F}E$, is well-defined and
continuous. Without losing of generality, we can assume that
$K_1=\overline{B(x_0,r)}$, for some $r>0$. Let
$K=\overline{B(x_0,r/4)}$. For every $x\in K$ and
$\chi\in\DD^*_{K-\{x_0\}}$, the function
$t\mapsto\chi_x(t)=\overline{\chi(t-x)}$ belongs to $\DD^*_{K_1}$
and thus
$\theta_{\Gamma}V_{\chi}f(x,\cdot)=\theta_{\Gamma}\mathcal{F}(\chi_x
f)\in\mathcal{F}E$. Fix $\chi\in\DD^*_{K-\{x_0\}}$. Our immediate
goal is to prove that the mapping $x\mapsto
\theta_{\Gamma}V_{\chi}f(x,\cdot)$, $K\rightarrow \mathcal{F}E$,
is continuous. Let $x'\in K$ be arbitrary but fixed. The Taylor
formula yields
\beqs
\overline{\chi(t-x)}-\overline{\chi(t-x')}=\sum_{|\beta|=1}(x'-x)^{\beta}\int_0^1\overline{\partial^{\beta}\chi(t-x'+s(x'-x))}ds.
\eeqs
When $x\in K$, the function
\beqs
t\mapsto\chi_{x,x',\beta}(t)=\int_0^1\overline{\partial^{\beta}\chi(t-x'+s(x'-x))}ds
\eeqs
belongs to $\DD^*_{K_1}$ and the set $\{\chi_{x,x',\beta}|\,
x\in K,\, |\beta|=1\}$ is bounded in $\DD^*_{K_1}$. Thus, there
exists $C'>0$ such that
\beqs
\|\theta_{\Gamma}\mathcal{F}(\chi_{x,x',\beta}
f)\|_{\mathcal{F}E}\leq C',\,\, \forall x\in K,\, \forall
|\beta|=1.
\eeqs
Since
\beqs
|\theta_{\Gamma}(\xi)V_{\chi}f(x,\xi)-\theta_{\Gamma}(\xi)V_{\chi}f(x',\xi)|=
\left|\sum_{|\beta|=1}(x'-x)^{\beta}\theta_{\Gamma}(\xi)\mathcal{F}(\chi_{x,x',\beta}f)(\xi)\right|,
\eeqs
for all $\xi\in\RR^d$, $x\in K$, the solidity of
$\mathcal{F}E$ proves
\beqs
\|\theta_{\Gamma}V_{\chi}f(x,\cdot)-\theta_{\Gamma}V_{\chi}f(x',\cdot)\|_{\mathcal{F}E}\leq
C_0C'd|x-x'|\rightarrow0,\,\, \mbox{as}\,\, x\rightarrow x',
\eeqs
which, in turn, verifies the continuity of $x\mapsto
\theta_{\Gamma}V_{\chi}f(x,\cdot)$, $K\rightarrow \mathcal{F}E$.
It remains to prove the continuity of the mapping
(\ref{vsrkln157}). Let $B$ be a bounded subset of
$\DD^*_{K-\{x_0\}}$. One easily verifies that $\{\overline{\chi_x}|\, x\in K,
\chi\in B\}$ is a bounded subset of $\DD^*_{K_1}$. As $\psi\mapsto
\theta_{\Gamma}\mathcal{F}(\psi f)$, $\DD^*_{K_1}\rightarrow
\mathcal{F}E$, is continuous and
$\theta_{\Gamma}V_{\bar{\chi}}f(x,\cdot)=\theta_{\Gamma}\mathcal{F}(\overline{\chi_x}
f)$, $\forall x\in K$, $\forall \chi\in \DD^*_{K-\{x_0\}}$, we infer that
the set $\{\theta_{\Gamma}V_{\bar{\chi}}f(x,\cdot)|\, \chi\in B,\, x\in
K\}$ is bounded in $\mathcal{F}E$ and consequently the image of
$B$ under the mapping (\ref{vsrkln157}) is bounded in
$\mathcal{C}(K;\mathcal{F}E)$. Since
$\DD^*_{K-\{x_0\}}$ is bornological, we conclude that
(\ref{vsrkln157}) is continuous.\\
$(iii)\Rightarrow (iv)$ is trivial and $(iv)\Rightarrow (i)$ follows easily by specialising $x=x_0$.
\end{proof}

\end{document}